\documentclass[12pt]{article}

\usepackage{amsthm,amsmath,amssymb}

\usepackage{graphicx}
\usepackage[T1]{fontenc}

\usepackage[colorlinks=true,citecolor=black,linkcolor=black,urlcolor=blue]{hyperref}


\newcounter{parentnumber}
\theoremstyle{plain}
\newtheorem{theorem}{Theorem}

\theoremstyle{lemma}
\newtheorem{lemma}{Lemma}

\theoremstyle{proposition}

\theoremstyle{corollary}
\newtheorem{corollary}{Corollary}

\theoremstyle{definition}
\newtheorem{definition}{Definition}

\theoremstyle{remark}
\newtheorem{remark}[theorem]{Remark}

 \title{   Almost every $n$-vertex graph is determined by its    $3 \log_2{n}$-vertex subgraphs \small{}  }

\author{Ameneh Farhadian \\
\\
\small Sharif University of Technology\\[-0.8ex]
\small   Tehran, I. R. Iran\\
\small\tt afarhadian@yahoo.com\\
\small\tt October 6, 2018\\ }

\date{ \small Mathematics Subject Classifications:  05C60}

\begin{document}

\maketitle

\begin{abstract}
The paper shows that almost every $n$-vertex graph is such that the multiset of its induced subgraphs on  $3 \log_2{n}$ vertices is sufficient to  determine  it up to isomorphism. 
Therefore,  for checking the isomorphism of a pair of $n$-vertex graphs, almost surely  the multiset of their $3 \log_2{n}$-vertex subgraphs is sufficient .

\end{abstract}

\section {Introduction}

We consider finite undirected graphs without loops and multiple edges. The set of vertices and the set of
edges of a graph $G$ are denoted by $V(G)$ and $E(G)$, respectively. The number of vertices of a graph $G$ is called its order. The neighbors of a vertex $v \in V(G)$  is denoted by  $N(v)$.

A bijection $f: V(G) \to  V(H)$ is called isomorphism from graph $G$ to graph $H$,  if $\lbrace x,y\rbrace \in E(G)$ if and only if $ \lbrace f(x),f(y)  \rbrace \in E(H)$. If graph $G$ is isomorphic to graph $H$, we denote it by $G \cong H$. An isomorphism $ f: V(G)  \to V(G)$ is called automorphism of
$G$, denoted by Aut($G$). A graph $G$  with trivial automorphism is called asymmetric, i.e. Aut($G$)=$I$.
We call the subset $S \subset V (G)$ is invariant under $Aut(G)$, if  $\theta(S)=S$ for any $\theta \in Aut(G) $.
A subgraph $S$ of a graph $G$ is a graph whose set of vertices and set of edges are all subsets of $G$. An induced subgraph is a subgraph which includes all possible edges. In this paper, any subgraph is an  induced subgraph.
 If $H$ is a subgraph of a graph $G$, $G \backslash V(H)$ is the induced subgraph on $V(G)-V(H)$.\\
In this paper, the set of all possible graphs on $n$ vertices is denoted by $\Gamma_n $. Assuming $G$ and $F$ are two arbitrary graphs, we define the multiplicity of  graph $F$ in graph $G$, denoted by $C_G(F)$, as the number of induced subgraphs of $G$  isomorphic with $F$. For instance, $C_G(K_2)$ denotes the number of edges of graph $G$.

Let $G$ be a graph with $n$ vertices and $S_m(G)$ be the multiset of all  $m$-vertex subgraphs of $G$. We call graph $G$ can be determined (or reconstructed), uniquely,  from $S_m(G)$, if $S_m(G)$ is unique among all graphs with $n$ vertices. That is, if for a graph $G'$ we have $S_m(G) = S_m(G')$, then $G \cong G'$.\\
%

 In the probability space of graphs on $n$ labeled vertices in which the edges are chosen independently, with probability $p=1/2$, we say that almost every graph $G$ has a property $Q$ if the probability that $G$ has $Q$ tends to 1 as $n \rightarrow \infty$.
We say that an assertion  holds almost surely, if the probability that it holds tends to 1 as $n$ goes to infinity.\\

 In \cite{MR0441789}, it is shown that almost every graph can be reconstructed uniquely from its $(n/2)$-vertex subgraphs. Here, this result has been improved  by showing that the subgraphs with $3 \log_2{n}$ vertices are sufficient.\\

In the following, the anchor of graph and shadow are defined. These concepts are defined in \cite{farhadian2016reconstruction, farhadian2017simple} to explain the reconstruction of a graph from its vertex-deleted subgraphs.

\begin{definition}
A proper induced subgraph $H$ of a graph $G$ is an anchor, if it occurs exactly once in graph $G$. \\
The shadow of a vertex $v \in V(G)-V(H)$ on the subgraph $H$ is  $N_v \cap V(H)$. We denote it by $s_{v,H}$ (or $s_v$ for simplicity, if $H$ is a
specified subgraph).\\
We call an anchor $H$ is  stable, if \textit{i)} the shadow of any vertex of $V(G) - V(H)$ on $H$ is invariant under Aut($H$) and \textit{ii)} there is no pair of vertices in $V(G)-V(H)$ whose shadows on $H$ are the same.

\end{definition}
Please note that the definition of  stable set is very close to the identification code which has been defined in \cite{MR1607726} and studied in \cite{MR2292538, MR2304730}. In definition of identification code $C$, the structure of the induced graph on $C$ is not considered. But, for a stable anchor $H$ the automorphism of $H$ is considered. In both of them, the vertices are identified by their neighborhood on $C$ or $H$.

\section{The main result }
As the main result of this paper, we show that the multiset of all induced subgraphs with $3 \log{n}$ vertices are  sufficient to determine the isomorphism class of almost every $n$-vertex graph. It means that for a randomly chosen graph $G$, this multiset is almost surely  unique among all graphs with $n$ vertices, i.e. there is no other graph $G'$  with the same multiset of subgraphs on $3 \log_2{n}$ vertices. Please note that in a multiset, the multiplicity of the members is taken into account.

\begin{theorem}\label{t1}
 Almost every $n$-vertex graph is such that the multiset of its  induced subgraphs with $3 \log_2{n}$ vertices is sufficient to determine it up to isomorphism.

In other words, the set of $\lbrace (F,C_G(F))\rbrace_{F \in \Gamma_{3\log_2{n}}} $ determines  almost every $n$-vertex graph $G$ up to isomorphism.

\end{theorem}
Furthermore, we will see that the multiplicity of  less  than ${n^2}/{2}$ subgraphs with at most $3 \log_2{n}$ vertices is sufficient to determine almost every $n$-vertex graph.
To prove the above theorem, we need the following lemmas.
\begin{definition}
We call subgraph $H'$ is 2-adjacent to subgraph $H$, if $\vert V(H')\vert=\vert V(H) \vert+2 $ and $V(H) \subset V(H')$, i.e. $H'$ is the induced subgraph on $V(H) \cup \lbrace v, w \rbrace$, for a pair of vertices $v,w \in V(G)-V(H)$. The multiset of all 2-adjacent subgraphs to subgraph $H$ is denoted by $N_2(H)$.
\end{definition}
 Clearly,  $\vert N_2(H) \vert={{\vert V(G)-V(H) \vert }\choose{2}}$. Assuming $A_H=\lbrace H \rbrace  \cup N_2(H)$, we show that if $H$ is a stable anchor, then $\lbrace(F, C_{G}(F)) \rbrace_{F \in A_H}$ determines $G$ up to isomorphism.

\begin{lemma}\label{asl}If  $H$  is  a stable anchor of a graph $G$, then the multiplicity of subgraph $H$ and all its 2-adjacent subgraphs in graph $G$, determines graph $G$ up to  isomorphism.

In other words,  $\lbrace(F, C_{G}(F)) \rbrace_{F \in A_H}$ determines $G$ up to isomorphism.

%
%
\end{lemma}
\begin{proof}
 Suppose that $H$ is a stable anchor and $H_{vw}$ denotes the induced subgraph on $ V(H) \cup \lbrace v, w \rbrace$ for $v,w \in V(G)-V(H)$. We show that the set of subgraphs $A_H=  \lbrace H \rbrace \cup \lbrace H_{vw} \vert v,w \in V(G)-V(H) \rbrace $ with their multiplicity in $G$, i.e $\lbrace (F,C_{G}(F)) \rbrace_{F \in A_H}$, are sufficient to determine $G$ up to isomorphism.\\
Since $H$ is an anchor,  we find out from $\lbrace (F,C_{G}(F)) \rbrace_{F \in A_H}$ that there is just one copy of $H$ in $G$. We want to complete subgraph $H$ to obtain graph $G$. We show that with the information that $\lbrace(F, C_{G}(F)) \rbrace_{F \in A_H}$ provides,  just a unique graph up to isomorphism can be obtained. \\
All subgraphs in $ \lbrace H_{vw} \vert v,w \in V(G)-V(H) \rbrace$ have $H$ as their subgraph, while there is just one copy of $H$ in $G$.  Thus, all copies of $H$  in all $H_{vw}$ are the same.\\
According to the definition of stable anchor, all vertices which are out of the anchor  have a unique neighbors set in $H$. Thus,   any vertex $v \in V(G)-V(H)$ out of anchor $H$, is identifiable by its neighbors on $H$. Due to $H_{vw}$, we add vertices $v$ and $w$  to subgraph $H$. Consequently, for any pair of vertices $v, w \in V(G)-V(H)$, $H_{v,w}$ determines the adjacency or non-adjacency of $v$ and $w$ in $G$. Therefore, using the information that $\lbrace (F, C_{G}(F)) \rbrace_{F \in A}$ provides, subgraph $H$ can be completed to graph $G$, uniquely.
\end{proof}
According to the above lemma,  graph $G$ with $n$ vertices and with stable anchor $H$ of order $k$ can be determined with the multiplicity of  just $1+ { {n-k}\choose {2}}$ subgraphs with at most $k+2$ vertices.\\


To warm up for the next lemma, we study the following problem.
Suppose that a graph $G$ with $n$ vertices and a graph $H$ with $3 \log{n}$ vertices are given. We want to know what is the probability that $H$ is a subgraph of $G$, that is $\mathbb{P} \lbrace C_G(H) \geq 1\rbrace $.

Let $\lbrace G_1, \cdots, G_N \rbrace $ be the set of all possible subgraphs of graph $G$ with, exactly, $m$ vertices. Clearly, $N={{n}\choose{m}}$. Let $A_i$ be the event that subgraph $G_i$ is isomrphic with $H$. 
  We have,
$$\mathbb{P} \lbrace   C_G(H) \geq 1 \rbrace =\mathbb{P}\lbrace A_1  \lor A_2  \lor \cdots \lor  A_N  \rbrace \leq \sum_{1 \leq i \leq N} \mathbb{P}\lbrace A_i  \rbrace $$
In addition, $\mathbb{P}\lbrace A_i  \rbrace = \mathbb{P}\lbrace G_i \cong H \rbrace = 1/ \vert \Gamma_m\vert $. where $\Gamma_m$ is the set of all possible graphs on $m$ vertices. We know $\vert \Gamma_m \vert \approx  { 2^{m^2/2}}/{ m!}$. Therefore,
$$\mathbb{P} \lbrace   C_G(H) \geq 1 \rbrace = \frac {{{n}\choose {m}}} {\vert \Gamma_m \vert} \approx \frac {n^m}{2^{m^2/2}}$$
Substituting $m$ with $3\log{n}$, we have
$$\mathbb{P} \lbrace   C_G(H) \geq 1 \rbrace \approx n^{-1.5 \log{n}}$$
Therefore, $\lim_{n \to \infty} \mathbb{P} \lbrace   C_G(H) \geq 1 \rbrace=0$. It means that for a given graph $H$ with $m=3 \log{n}$ vertices, almost all graphs with $n$ vertices do not have any subgraph isomorphic with $H$.
\begin{lemma}\label{prob}
For almost every $n$-vertex graph $G$,  an  arbitrary subgraph with $3 \log_2{n}$ vertices is a stable anchor.


\end{lemma}

\begin{proof}

Let  $m=3 \log_2{n}$ and $H$ be an arbitrary $m$-vertex subgraph of a $n$-vertex graph $G$. Since $H$ is an arbitrary graph and almost every graph is asymmetric \cite{MR0156334}, $H$ is asymmetric, almost surely.\\
Now, we want to check what is the probability of emerging another copy of $H$ in $G$, provided that there exists a copy of $H $ in $G$, that is  $$ \mathbb{P} \lbrace  C_G(H)>1\vert C_G(H) \geq 1 \rbrace$$ We show that $\lim_{n \to \infty}  \mathbb {P}\lbrace C_G(H)>1 \vert C_G(H) \geq 1 \rbrace =0$. It means that almost surely there is not any copy of $H$, except itself, and consequently, subgraph $H$ is an anchor for almost every graph $G$.\\
 Now, assuming that $H$ (with $m=3 \log{n}$ vertices) is a subgraph of $G$ (with $n$ vertices), we arrange $m$-vertex subgraphs with respect to the number of vertices which they share with $H$. For instance, there are ${{m}\choose {1}}  {{n-m}\choose {1}} $ subgraphs which  share ($m-1$) vertices with $H$. Let $F$ be a subgraph of $G$ which share ($m-1$) vertices with $F$ and $\lbrace w \rbrace=V(F)-V(H)$. The induced subgraph on ($m-1$) vertices is common in both graphs $F$ and $H$. Thus, the state of any pair of vertices in $V(H)\cap V(F)$, i.e. $(m-1)(m-2)/2$  pair of vertices, are determined by $H$. Just $m(m-1)/2- (m-1)(m-2)/2$(=$m-1$) pairs of vertices, i.e.those pairs of vertices which include $w$,  must be determined by probability. Thus, the probability that $F$ is isomorphic with $H$ is $2^{-(m-1)}$.   \\
As we mentioned earlier, almost every graph is asymmetric. Thus, we ignore any asymmetry in $H$ or its subgraphs. \\
Suppose that subgraph $F$ with $m$ vertices  share ($m-i$) vertices with subgraph $H$.  If $\vert V(F) \cap V(H)\vert= m-i$, there are $i!$  possibility for mapping of $V(H)-(V(H) \cap V(F))$ onto $ V(F)-(V(H) \cap V(F))$. We show that for each mapping $\phi$, the probability of isomorphism of $H$ onto $F$ by $\phi$  is $2^{-(m^2-(m-i)^2)/2}$.\\
 Each isomorphism mapping $\phi $ of $H$ onto $F$, determines adjacency or non-adjacency of  pair of vertices in $V(F)$. The number of pair of vertices of a graph with $m$ vertices is approximated by $m^2/2$, here. Graph $F$ with $m$ vertices has approximately $m^2/2$  pair of vertices in which $(m-i)^2/2$ of them belong to the intersection of $V(H)$ and $ V(F)$ and, consequently,  are determined by graph $H$. Thus, for each mapping $\phi$,  we have $m^2/2-(m-i)^2/2$ pair of vertices in which the adjacency or non-adjacency should be determined by probability. Thus, the probability of isomorphism of $F$ onto $H$  by mapping $\phi$ is $2^{-(m^2-(m-i)^2)/2}$ . The number of subgraphs which share ($m-i$) vertices with $H$ in $G$ is ${{m}\choose {i}}  {{n-m}\choose {i}} $. Therefore, assuming $G_1 \cong H$ we have
$$ \mathbb {P} \lbrace C_G(H)>1 \vert C_G(H) \geq 1 \rbrace= \mathbb{P}\lbrace A_2  \lor A_3  \lor \cdots \lor  A_N  \vert A_1 \rbrace$$
$$\leq \sum_{2 \leq t \leq N} \mathbb{P}\lbrace A_t  \vert A_1 \rbrace $$
$$=\sum_{i=1}^{m} {{m}\choose {i}}  {{n-m}\choose {i}}.  i! .2^{-(m^2/2-(m-i)^2/2)} $$
$$\leq \sum_{i=1}^{m} {{m}\choose {i}}  (n^i/i!).  i! .2^{-mi+i^2/2} $$
Since $i \leq m$, we have $2^{i^2/2} \leq 2^{mi/2}$. Thus, $2^{-mi+i^2/2}\leq 2^{-mi+mi/2}=2^{-mi/2}$. Therefore,
$$ \mathbb{P} \lbrace C_G(H)>1 \vert C_G(H) \geq 1 \rbrace \leq \sum_{i=1}^{m} {{m}\choose {i}}  n^i .2^{-mi/2} $$
Since $\frac{n}{2^{m/2}}(=\frac{1}{\sqrt{n}})<1$, we have
$$ \leq \sum_{i=0}^{\infty} {{m}\choose {i}}  n^i .2^{-mi/2}-1=(1+\frac {n}{2^{m/2}})^m-1 $$
We know that
$$\lim_{n \to \infty } (1+\frac {n}{2^{m/2}})^m =\lim_{n \to \infty } (1+\frac {1}{\sqrt{n}})^{3\log{n}} =1$$
Therefore,
$$ \lim_{n \to \infty } \mathbb{P} \lbrace C_G(H)>1 \vert C_G(H) \geq 1 \rbrace \leq 0$$
Thus,
$$ \lim_{n \to \infty }\mathbb{P}\lbrace C_G(H) >1 \vert C_G(H) \geq 1 \rbrace = 0$$
It means that  for almost every $G$ with $n$ vertices, if $H$ is an arbitrary  subgraph  of $G$ with $3 \log{n}$ vertices, then there is not any copy of $H$ in $G$, except itself. In other words,  for almost every graph with $n$ vertices, an arbitrary subgraph with $3\log{n}$ vertices is an anchor.

 Now, we should prove that,  almost surely,  $H$ is a stable anchor. That is, the probability that all shadows on $H$ are invariant under $Aut(H)$ and no pair of vertices in $V(G)-V(H)$ have the same shadows on $H$, tends to 1 as $n$ goes to infinity. At first, we mentioned that almost surely $H$ is asymmetric, thus, all shadows on $H$ are invariant under $Aut(H)$. It is sufficient to show that, almost surely,  there is no pair of vertices in $V(G)-V(H) $ whose shadows on $H$ are the same.  There are   $2^{3\log_2{n}}$ (= $n^3$) subsets for the set $V(H)$. Thus, the number of possible different shadows on $H$ is $n^3$. The anchor $H$ has $3 \log {n}$ vertices. Thus, there are ($n-3\log{n}$) vertices out of anchor $H$.  The probability that each vertex out of $H$ chooses a unique shadow on $H$ is $$p_n= \frac {{{n^3}\choose {n-3 \log_2{n}}}(n-3\log_2{n})!}{(n^{3})^{n-3\log_2{n}} }$$
$$\approx   n^3\cdots  (n^3-(n-3\log_2{n}-1))/ (n^{3(n-3\log_2{n})})$$ Thus, we have $p_n  \approx 1$  for large $n$. It means that, almost surely, no pair of vertices in $V(G)-V(H)$ share the same shadow on $H$. Thus, $H$ is a stable anchor.

\end{proof}

\begin{remark}\label{rem_c}
Needless to say, the lemma remains true if $3 \log_2{n}$ is replaced by $3 \log_2{n}+c$ where $c$ is a fixed integer.
\end{remark}
\begin{lemma}\label{kel}
If $H$ is a graph with at most $m$ vertices, then the multiset of $m$-vertex subgraphs of a graph $G$  determines the number of occurrence of $H$ in $G$.
\end{lemma}
In other words, the above lemma states that if $H$ is graph with at most $m$ vertices, then $C_G(H)$ can be determined by $\lbrace (F,C_G(F))\rbrace_{F \in \Gamma_m} $.
The above lemma is a generalization of Kelly's lemma \cite{MR0087949}. Kelly's lemma  states that the number of occurrence of any arbitrary proper subgraph of a graph $G$ with $n$ vertices  is inferable from the set of its subgraphs with ($n-1$) vertices.
\begin{proof}
 We want to show that the number of copies of a  $k$-vertex subgraph $H$, i.e. $C_G(H)$, such that $k<m$ can be determined by $\lbrace (F,C_G(F))\rbrace_{F \in \Gamma_m} $. The number of $m$-vertex subgraphs of $G$ which include $V(H)$ is  ${{n-k} \choose {m-k}}$. Thus, each copy of subgraph $H$ is repeated exactly in  ${{n-k} \choose {m-k}}$ $m$-vertex subgraphs. Consequently, the number of copies of  $H$ in $G$ is
$$C_G(H)=\frac{\sum_{F \in \Gamma_m}C_G(F)}{{{n-k} \choose {m-k}}} $$

\end{proof}

\textit{Proof of  Theorem \ref{t1}}:

According to Lemma \ref{prob} and Remark \ref{rem_c}, for almost every $n$-vertex graph $G$, an arbitrary subgraph $H$ with $3 \log_2{n}-2$ vertices is a stable anchor.
 Since $ H$  is a stable anchor, the multiplicity of  subgraphs $H$ and $H_{v,w}$ where $v,w \in V(G)-V(H)$ are sufficient to determine $G$, due to  Lemma \ref{asl}. The subgraph
$H$ has $3 \log_2{n}-2$ vertices and $H_{v,w}$'s have $3 \log_2{n}$ vertices. According to Lemma \ref{kel}, subgraphs with $3 \log_2{n}$ vertices are sufficient to know the number of any subgraph with smaller order. Therefore, subgraphs with $3 \log_2{n}$ vertices are sufficient to determine graph $G$ up to isomorphism.$\square$\\

\begin{corollary}\label{c}
To decide the isomorphism of a pair of  $n$-vertex graphs $G_1$ and $G_2$, almost surely, it is sufficient to compare their multiset of   subgraphs  on $3\log_2{n}$ vertices, i.e. if $ C_{G_1}(F)=C_{G_2}(F)$, for any $F \in \Gamma_{3 \log{n}}$, then $G_1 \cong G_2$.
\end{corollary}

\begin{corollary} Almost every $n$-vertex  graph can be reconstructed from subgraphs with just $3 \log_2{n}$ vertices.
\end{corollary}

%
%

\bibliographystyle{plain}


\end{document}